\documentclass[12pt]{article}
\usepackage[english]{babel}
\usepackage[all]{xy}
\usepackage{fancyhdr}
\usepackage{setspace, amsmath, amsthm, amssymb, amsfonts, amscd, epic, graphicx, ulem, dsfont}
\usepackage{multirow}
\usepackage{amsthm}
\usepackage[T1]{fontenc}
\usepackage{answers}
\usepackage{hyperref}
\usepackage{setspace}
\newtheorem{theorem}{Theorem}

\newtheorem{proposition}[theorem]{Proposition}
\newtheorem{corollary}[theorem]{Corollary}
\newtheorem{example}[theorem]{Example}

\newtheorem{remark}[theorem]{Remark}

\makeatletter

\@addtoreset{equation}{section}
\makeatother

\title{Construction of normally biharmonic submanifolds}

\author{Ahmed Mohammed Cherif\footnote{University Mustapha Stambouli Mascara, Faculty of Exact Sciences, Mascara 29000, Algeria. Email:
a.mohammedcherif@univ-mascara.dz
}}
\date{}
\begin{document}
\maketitle

\begin{abstract}
We examine biharmonic submanifolds within warped product structures. For a submanifold $(M,g)\subset (N,h)$ and a positive smooth function $f:I\to\mathbb{R}^+$, we study the inclusion $\varphi:(I\times M,\widetilde{g})\to (I\times N,\overline{h})$, where $\widetilde{g}=dt^2+f^2g$ and $\overline{h}=dt^2+f^2h$. We relate the tension and bitension fields of $\varphi$ to the warping function and the geometry of $M$. We further characterize tangentially and normally biharmonic cases via differential conditions on $f$, and interpret these conditions in terms of the Ricci curvature of $(M,g)$ and $(I\times M,\widetilde{g})$.
\end{abstract}

\begin{flushleft}
Subjclass: 53C20; 58E20.\\
\end{flushleft}

\maketitle

\section{Introduction}

The energy functional of a smooth map $\varphi:(M,g)\longrightarrow(N,h)$ between Riemannian manifolds is defined by
\begin{equation}\label{eq1.1}
    E(\varphi)=\frac{1}{2}\int_D |d\varphi|^2 \, v^g,
\end{equation}
where $D$ is a compact domain of $M$, $|d\varphi|$ denotes the Hilbert--Schmidt norm of the differential $d\varphi$, and $v^g$ is the volume element on $(M,g)$. A map $\varphi$ is called harmonic if it is a critical point of the energy functional (\ref{eq1.1}). The associated Euler--Lagrange equation is given by (see \cite{BW,ES,YX})
\begin{equation}\label{eq1.2}
\tau(\varphi)\equiv\operatorname{trace}\nabla d\varphi
=\sum_{i=1}^m\left[ \nabla^\varphi_{e_i} d\varphi(e_i)
- d\varphi(\nabla^M_{e_i}e_i)\right]=0,
\end{equation}
where $\{e_i\}_{i=1}^m$ is a local orthonormal frame on $(M,g)$, $\nabla^{M}$ is the Levi--Civita connection of $(M,g)$, $\nabla^{\varphi}$ is the pull-back connection on $\varphi^{-1}TN$.\\
A natural generalization of harmonic maps is obtained by considering the square norm of the tension field. More precisely, the bienergy functional of $\varphi\in C^\infty(M,N)$ is defined by
\begin{equation}\label{eq1.3}
    E_2(\varphi)=\frac{1}{2}\int_D |\tau(\varphi)|^2 \, v^g.
\end{equation}
A map $\varphi$ is called biharmonic if it is a critical point of the bienergy functional, that is, if it satisfies the Euler--Lagrange equation associated to (\ref{eq1.3}), given by
\begin{eqnarray}\label{eq1.4}
\tau_2(\varphi)
   &\equiv& -\operatorname{trace} R^N(\tau(\varphi),d\varphi)d\varphi
           -\operatorname{trace}\, (\nabla^{\varphi})^2 \tau(\varphi) \nonumber \\
   &=& -\sum_{i=1}^m R^N(\tau(\varphi),d\varphi(e_i))d\varphi(e_i)
       -\sum_{i=1}^m \nabla^{\varphi}_{e_i}\nabla^{\varphi}_{e_i}\tau(\varphi) \nonumber \\
   &&+\sum_{i=1}^m \nabla^{\varphi}_{\nabla^M_{e_i}e_i}\tau(\varphi)=0,
\end{eqnarray}
where $R^N$ is the curvature tensor of $(N,h)$ defined by
\[
R^N(X,Y)Z=\nabla^N_X \nabla^N_Y Z-\nabla^N_Y \nabla^N_X Z-\nabla^N_{[X,Y]}Z,
\]
for $X,Y,Z\in\Gamma(TN)$ (see \cite{Jiang,YX}). It is clear from (\ref{eq1.4}) that every harmonic map is biharmonic, non-harmonic biharmonic maps are called proper biharmonic maps.\\
Let $M$ be a submanifold of a Riemannian manifold $(N,\langle\cdot,\cdot\rangle)$ of dimension $m$, and let $\mathbf{i}:M\hookrightarrow N$ denote the canonical inclusion. Let $\{e_i\}_{i=1}^m$ be a local orthonormal frame with respect to the induced metric $g$. We denote by $\overline{\nabla}$ (resp. $\nabla$) the Levi--Civita connections of $(N,\langle\cdot,\cdot\rangle)$ (resp. $(M,g)$), by $B$ the second fundamental form, and by $H=(1/m)\operatorname{trace}\,B$ the mean curvature vector field (see \cite{BW,ON}). The submanifold $(M,g)$ is called harmonic (resp. biharmonic) if $\tau(\mathbf{i})=0$ (resp. $\tau_2(\mathbf{i})=0$) (see \cite{chen}). Moreover, the tension and bitension fields of the inclusion satisfy
\begin{equation}\label{eq1.5}
    \tau(\mathbf{i})=mH,\quad
    \tau_2(\mathbf{i})=-m\sum_{i=1}^m\left\{
    \overline{R}(H,e_i)e_i+
    \overline{\nabla}_{e_i}\overline{\nabla}_{e_i}H
    -\overline{\nabla}_{\nabla_{e_i}e_i}H\right\},
\end{equation}
where $\overline{R}$ is the curvature tensor of $(N,h)$. The bitension field admits a natural decomposition into tangential and normal components (see \cite{chen}). In \cite{OU}, Ye-Lin Ou proved that a hypersurface $(M,g)$ in a Riemannian manifold $(N,\langle ,\rangle)$ with mean curvature vector field $H= \lambda \eta $, that is the dimension of $N$ is $m+1$, is biharmonic if and only if
\begin{equation}\label{S}
\left\{
\begin{array}{lll}
\tau_2(\mathbf{i})^\bot\equiv m\left[-\Delta(\lambda)  + \lambda |A|^2 -\lambda \operatorname{\overline{Ric}}(\eta , \eta )\right]\eta  &=& 0; \\\\
\tau_2(\mathbf{i})^\top\equiv m\left[2A(\operatorname{grad} \lambda)+ m\lambda \operatorname{grad}\lambda -2 \lambda (\operatorname{\overline{Ricci}} \eta)^\top\right]  &=& 0,
\end{array}
\right.
\end{equation}
where $\operatorname{\overline{Ric}} $ (resp. $\operatorname{\overline{Ricci}}$) is the Ricci curvature (resp. Ricci tensor) of $(N,\langle ,\rangle )$, $\lambda$ denote the mean curvature function of $(M,g)$, and $A$ the shape operator with respect to the unit normal vector field $\eta$.\\
A submanifold is said to be tangentially (resp. normally) biharmonic if its tangential (resp. normal) component of the bitension field vanishes. For the tangentially biharmonic submanifold see for example \cite{S}.\\
Warped product manifolds offer a versatile framework in differential geometry and mathematical physics for constructing spaces with prescribed curvature (see \cite{ON}). Given a positive smooth function $f:I\subset\mathbb{R}\to\mathbb{R}^+$, consider the warped products $(I\times M,\widetilde{g})$ and $(I\times N,\overline{h})$, with $\widetilde{g}=dt^2+f^2 g$ and $\overline{h}=dt^2+f^2 h$. We study the natural inclusion $\varphi:(I\times M,\widetilde{g}) \to (I\times N,\overline{h}),$ $(t,x)\mapsto (t,x),$
as a model for analyzing the interaction between biharmonicity and warped product geometry.
Our goal is to relate the tension and bitension fields of $\varphi$ to the warping function $f$ and the geometry of the submanifold $(M,g)\subset (N,h)$. We derive explicit expressions for $\tau(\varphi)$ and $\tau_2(\varphi)$ in terms of $f$, its derivatives, and the mean curvature vector $H$, leading to rigidity results and differential characterizations of the inclusion map.\\
In particular, we obtain a formula for $\overline{h}(\tau_2(\varphi),\tau(\varphi))$ involving derivatives of $f$ and $|H|^2$, establishing a direct connection between biharmonicity and the warped structure. As applications, we show that biharmonicity reduces to harmonicity under suitable conditions and characterize normally biharmonic behavior via simple constraints on $f$.
Finally, we express these conditions in terms of curvature, providing an equivalent formulation involving the Ricci tensors of $(M,g)$ and $(I\times M,\widetilde{g})$. This yields a geometric interpretation in terms of curvature compatibility between the base manifold and the warped product.

\section{Main Results}

According to (\ref{S}), we obtain the following.

\begin{theorem}\label{th01}
A hypersurface $(M,g)$ immersed in a space form $(N(c),h)$ is normally biharmonic if and only if its mean curvature function $\lambda$ satisfies
\[
\Delta \lambda = \left(|A|^2 - mc\right)\lambda,
\]
where $m=\dim M$.
\end{theorem}

\begin{remark}
In view of \eqref{S}, any constant mean curvature (CMC) hypersurface immersed in a space form is normally biharmonic if and only if it is biharmonic.
\end{remark}

\begin{example}
Consider a surface $M$ in $\mathbb{S}^3$ parametrized by
\[
X(u,v) = (u,v,r),
\]
where $r\neq 0$ is a constant. The induced metric on $M$ is
\[
g = \frac{4}{\big(1+r^2+u^2+v^2\big)^2}\,(du^2 + dv^2).
\]
A direct computation shows that the mean curvature function is constant and given by $\lambda = r$. The associated unit normal vector field along $M$ in $\mathbb{S}^3$ is
\[
\eta = \frac{1}{2}\big(1+r^2+u^2+v^2\big)\,\frac{\partial}{\partial z}.
\]
Furthermore, one finds that $\Delta \lambda = 0$ and $|A|^2 = 2r^2$.
Substituting these into the equation of Theorem~\ref{th01} with $c=1$, we obtain the condition
$r^2 - 1 = 0$. Therefore, for $r=\pm 1$, the surface $M$ is normally biharmonic in $\mathbb{S}^3$.
\end{example}

We next present an example of a hypersurface whose mean curvature function is not constant.

\begin{example}
Let $M$ be a surface of revolution in $\mathbb{R}^3$ parametrized by
\[
X(u,v) = \big(r u \cos v,\; r u \sin v,\; u\big),
\]
where $r>0$ is a constant, $u>0$, and $v \in (0,2\pi)$. The induced metric on $M$ is given by
\[
g = (1+r^2)\,du^2 + r^2 u^2\,dv^2.
\]
A direct computation shows that the mean curvature function $\lambda=\lambda(u)$ takes the form
\[
\lambda(u) = \frac{1}{2r\sqrt{1+r^2}}\,\frac{1}{u}.
\]
The corresponding unit normal vector field along $M$ in $\mathbb{R}^3$ is
\[
\eta = -\frac{\cos v}{\sqrt{1+r^2}}\,\frac{\partial}{\partial x}
       -\frac{\sin v}{\sqrt{1+r^2}}\,\frac{\partial}{\partial y}
       +\frac{r}{\sqrt{1+r^2}}\,\frac{\partial}{\partial z}.
\]
Moreover, one computes
\begin{align*}
\Delta \lambda &= \frac{1}{2r(1+r^2)^{3/2}}\,\frac{1}{u^3}, \\
|A|^2 &= \frac{1}{r^2(1+r^2)}\,\frac{1}{u^2}.
\end{align*}
Substituting these expressions into the equation of Theorem~\ref{th01} with $c=0$, we obtain the condition
$r^2 - 1 = 0$. Hence, for $r=1$, the surface $M$ is normally biharmonic in $\mathbb{R}^3$, and its mean curvature function is non-constant.
\end{example}

\textbf{Problem.} Classify all normally biharmonic surfaces in a $3$-dimensional space form \(N(c)\), with particular emphasis on the cases of \(\mathbb{R}^3\), \(\mathbb{S}^3\), and \(\mathbb{H}^3\).\\


Let $(M,g)$ be a submanifold in a Riemannian manifold \((N, h)\).
Given any smooth positive function $f$ defined on an open interval $I \subset \mathbb{R}$, we consider the inclusion map
\begin{align*}
    \varphi: (I \times M, \widetilde{g}) &\rightarrow (I \times N, \overline{h}), \\
    (t,x) &\mapsto (t,x)
\end{align*}
where the Riemannian metrics $\widetilde{g}$ and $\overline{h}$ are defined by $\widetilde{g} = dt^2 + f^2 g$ and $\overline{h} = dt^2 + f^2 h$, respectively.
In the context of biharmonic submanifolds, it is important to understand how their geometric properties interact with warped product structures. By considering a biharmonic submanifold $(M,g)$ in a Riemannian manifold $(N,h)$ and analyzing the inclusion map into an associated warped product space, we can derive explicit formulas relating the bitension and tension fields. The following theorem establishes such a relation, expressing it in terms of the warping function $f$ and the mean curvature vector $H$.

\begin{theorem}\label{th1}
 Let $(M,g)$ be a biharmonic submanifold in $(N,h)$. Then
\begin{eqnarray*}
 \overline{h}(\tau_2(\varphi),\tau(\varphi))
   &=& \frac{2m^2[ff''+(m-1)(f')^2]}{f^4}|H|^2.
\end{eqnarray*}
\end{theorem}

\begin{proof}
We compute the tension field of the map $\varphi$. Let $\{e_i\}_{i=1}^m$ be a local orthonormal frame field on $(M,g)$. Hence, $\{\partial_t,\frac{1}{f}e_i\}_{i=1}^m$ is an orthonormal frame field on $(I\times M,\widetilde{g})$. We have
\begin{eqnarray}\label{t-1}
 \tau(\varphi)
   &=&\nonumber (\nabla d\varphi)(\partial_t,\partial_t)+\frac{1}{f^2}\sum_{i=1}^m(\nabla d\varphi)(e_i,e_i) \\
   &=&\nonumber  \nabla^\varphi_{\partial_t}d\varphi(\partial_t)-d\varphi(\widetilde{\nabla}_{\partial_t}\partial_t)
        +\frac{1}{f^2}\sum_{i=1}^m\Big[\nabla^\varphi_{e_i}d\varphi(e_i)-d\varphi(\widetilde{\nabla}_{e_i}e_i)\Big]\\
   &=& \overline{\nabla}_{\partial_t}\partial_t-\widetilde{\nabla}_{\partial_t}\partial_t
   +\frac{1}{f^2}\sum_{i=1}^m\Big[\overline{\nabla}_{e_i}e_i-\widetilde{\nabla}_{e_i}e_i\Big].
\end{eqnarray}
Here, $\widetilde{\nabla}$ (resp. $\overline{\nabla}$) denote the Levi-Civita connection of $(I\times M,\widetilde{g})$
(resp. of $(I\times N,\overline{h})$). By using the formulas of the Levi-Civita connection of warped product manifold
\begin{eqnarray*}
   \overline{\nabla}_{\partial_t}\partial_t&=&0,  \\
   \widetilde{\nabla}_{\partial_t}\partial_t&=&0,  \\
   \overline{\nabla}_{U}V &=& \nabla_{X}^NY-h(U,V)ff'\partial_t,\\
   \widetilde{\nabla}_{X}Y &=& \nabla_{X}^MY-g(X,Y)ff'\partial_t,
\end{eqnarray*}
where $X,Y\in\Gamma(TM)$ and $U,V\in\Gamma(TN)$ (see \cite{ON}), the equation (\ref{t-1}) becomes
\begin{eqnarray}\label{t-2}
 \tau(\varphi)
   &=&\nonumber \frac{1}{f^2}\sum_{i=1}^m\Big[\nabla^N_{e_i}e_i-\nabla_{e_i}e_i\Big]\\
   &=&\nonumber\frac{1}{f^2}\sum_{i=1}^mB(e_i,e_i)\\
   &=&\frac{m}{f^2}H.
\end{eqnarray}
Now, we compute the bitension field of the smooth map $\varphi$. We have
\begin{eqnarray}\label{t-3}
 \tau_2(\varphi)
   &=&\nonumber -\overline{R}(\tau(\varphi),\partial_t)\partial_t-\frac{1}{f^2}\sum_{i=1}^m\overline{R}(\tau(\varphi),e_i)e_i\\
   & &\nonumber -\Big[\nabla^\varphi_{\partial_t}\nabla^\varphi_{\partial_t}\tau(\varphi)-\nabla^\varphi_{\widetilde{\nabla}_{\partial_t}\partial_t}\tau(\varphi)\Big]\\
   & & -\frac{1}{f^2}\sum_{i=1}^m\Big[\nabla^\varphi_{e_i}\nabla^\varphi_{e_i}\tau(\varphi)-\nabla^\varphi_{\widetilde{\nabla}_{e_i}e_i}\tau(\varphi)\Big],
\end{eqnarray}
where $\overline{R}$ is the curvature tensor of $(I\times N,\overline{h})$. By using the formulas of the curvature tensor of warped product manifold
\begin{eqnarray*}
  \overline{R}(U,\partial_t)\partial_t &=& -\frac{f''}{f} U,\\
  \overline{R}(V,W)U&=& R^N(V,W)U-(f')^2\big[ h(U,W)V-h(U,V)W \big],
\end{eqnarray*}
where $U,V,W\in\Gamma(TN)$ (see \cite{ON}), and equation (\ref{t-2}), we get the following
\begin{eqnarray}\label{t-4}
 -\overline{R}(\tau(\varphi),\partial_t)\partial_t-\frac{1}{f^2}\sum_{i=1}^m\overline{R}(\tau(\varphi),e_i)e_i
   &=&\nonumber\frac{m[ff''+m(f')^2]}{f^4}H\\
   & &-\frac{m}{f^4}\sum_{i=1}^mR^N(H,e_i)e_i.
\end{eqnarray}
According to simple calculations, we obtain
\begin{eqnarray}\label{t-5}
-\Big[\nabla^\varphi_{\partial_t}\nabla^\varphi_{\partial_t}\tau(\varphi)-\nabla^\varphi_{\widetilde{\nabla}_{\partial_t}\partial_t}\tau(\varphi)\Big]
   &=& \frac{m[ff''-2(f')^2]}{f^4}H.
\end{eqnarray}
and the following
\begin{eqnarray}\label{t-6}
   -\frac{1}{f^2}\sum_{i=1}^m\Big[\nabla^\varphi_{e_i}\nabla^\varphi_{e_i}\tau(\varphi)-\nabla^\varphi_{\widetilde{\nabla}_{e_i}e_i}\tau(\varphi)\Big]
   &=&\nonumber-\frac{m}{f^4}\sum_{i=1}^m\Big[\nabla^N_{e_i}\nabla^N_{e_i}H-\nabla^N_{\nabla_{e_i}e_i}H\Big]\\
   &&+\frac{m^2(f')^2}{f^4}H-\frac{m^2f'}{f^3}|H|^2\partial_t,
\end{eqnarray}
where $|H|^2=h(H,H)$. Substituting (\ref{t-4})-(\ref{t-6}) in (\ref{t-3}), we get
\begin{eqnarray}\label{t-7}
 \tau_2(\varphi)
   &=& \frac{2m[ff''+(m-1)(f')^2]}{f^4}H+\frac{m}{f^4}\tau_2(\mathbf{i})
   -\frac{m^2f'}{f^3}|H|^2\partial_t.
\end{eqnarray}
The Theorem \ref{th1} follows from (\ref{t-7}) and the biharmonicity condition of
the submanifold $(M,g)$ in $(N,h)$ with $\overline{h}(\partial_t,\tau(\varphi))=0$.
\end{proof}

The following corollary reveals a rigidity result. Under certain conditions, the biharmonicity of an inclusion map into a product manifold implies harmonicity, thus narrowing the class of possible biharmonic maps in this context.

\begin{corollary}\label{co1}
Let $(M,g)$ be a biharmonic submanifold in $(N,h)$. Then, the inclusion map
$
\varphi: (I \times M, \widetilde{g}) \longrightarrow (I \times N, \overline{h})
$
is biharmonic if and only if it is harmonic.
\end{corollary}

A map is said to be biconservative if the divergence of its bienergy stress-energy tensor vanishes \cite{LMO}. For immersions, this condition is equivalent to the vanishing of the tangential component of the bitension field (also called tangentially biharmonic \cite{S}). According to equation (\ref{t-7}), we get the following results.

\begin{corollary}
Let $(M,g)$ be a biharmonic submanifold of a Riemannian manifold $(N,h)$. Then,
the tangential component of the bitension field $\tau_2(\varphi)$ vanishes at a point $(t_0,x_0)\in I\times M$ if and only if $f'(t_0) = 0$.
\end{corollary}

For normally biharmonic hypersurface, we obtain the following.

\begin{corollary}
Let $(M,g)$ be a hypersurface of $(N,h)$.
Then, the hypersurface $(I \times M,\widetilde{g})$ is normally biharmonic if and only if the warping function $f$ is of the form
$
f(t) = (a t + b)^{1/m},
$
for some constants $a,b \in \mathbb{R}$, and we have $$\Delta(\lambda) = \lambda |A|^2 -\lambda \operatorname{Ric}^N(\eta , \eta ),$$
where $\lambda$ denote the mean curvature function of $(M,g)$ in $(N,h)$.
\end{corollary}

This corollary provides a precise characterization of the warping functions for which a natural geometric compatibility condition holds between the tension and bitension fields of the inclusion map.

\begin{corollary}\label{co2}
Let $(M,g)$ be a biharmonic submanifold of $(N,h)$. Consider the inclusion map
$
\varphi:(I \times M,\widetilde{g}) \longrightarrow (I \times N,\overline{h}).
$
Then, the tension field and the bitension field of $\varphi$ are orthogonal, that is the submanifold $(I \times M,\widetilde{g})$ is normally biharmonic if and only if the warping function $f$ is of the form
$
f(t) = (a t + b)^{1/m},
$
for some constants $a,b \in \mathbb{R}$.
\end{corollary}

\begin{example}
For every $m \geq2$, the inclusion map
$$\varphi : (I \times \mathbb{S}^m(\tfrac{1}{\sqrt{2}}), \widetilde{g}) \longrightarrow (I \times \mathbb{S}^{m+1}, \overline{h}),$$
is normally biharmonic, where $f$ is of the form
$
f(t) = (a t + b)^{1/m},
$
for some constants $a,b \in \mathbb{R}$. This follows from the fact that the $\mathbb{S}^m(\tfrac{1}{\sqrt{2}}) \subset \mathbb{S}^{m+1}$ is a constant mean curvature function proper biharmonic hypersurface.
\end{example}


This corollary gives a pointwise characterization for the vanishing of the bitension field of the inclusion map in the warped product setting.
\begin{corollary}
Let $(M,g)$ be a proper biharmonic submanifold of a Riemannian manifold $(N,h)$. Then, $\tau_2(\varphi)=0$
on $\{t_0\}\times M$ if and only if $f'(t_0) = 0$ and $f''(t_0) = 0$.
\end{corollary}


We denote by $\operatorname{Ric}$ (resp. $\widetilde{\operatorname{Ric}}$) the Ricci curvature of $(M,g)$ (resp. of $(I \times M, \widetilde{g})$). The following result provides an explicit relation between the bitension and tension fields of the inclusion map $\varphi$ in terms of the Ricci curvatures of $(M,g)$ and the warped product manifold.
\begin{proposition}\label{pro1}
 Let $(M,g)$ be a biharmonic submanifold in $(N,h)$. Then
\begin{eqnarray*}
 \overline{h}(\tau_2(\varphi),\tau(\varphi))
   &=& \frac{2m^2}{f^4}\left[\operatorname{Ric}(X,X)-\widetilde{\operatorname{Ric}}(X,X)\right]|H|^2.
\end{eqnarray*}
\end{proposition}
\begin{proof}
Note that, for each unit vector field  $X$ on $(M,g)$, we have
\begin{equation}\label{t-8}
\widetilde{\operatorname{Ric}}(X,X)=\operatorname{Ric}(X,X)-[ff''+(m-1)(f')^2],
\end{equation}
(see \cite{ON}). The proof of Proposition \ref{pro1} follows from  Theorem  \ref{th1} and (\ref{t-8}).
\end{proof}

According to Corollary \ref{co2} and Proposition \ref{pro1}, we deduce the following.

\begin{corollary}
Let $(M,g)$ be a biharmonic submanifold of $(N,h)$. Consider the inclusion map
$
\varphi:(I \times M,\widetilde{g}) \longrightarrow (I \times N,\overline{h}).
$
Then, the submanifold $(I \times M,\widetilde{g})$ is normally biharmonic if and only if
$\widetilde{\operatorname{Ric}}|_{TM}=\operatorname{Ric}$.
\end{corollary}


\end{document}